\documentclass[12pt]{article}
\usepackage{latexsym}
\usepackage{amsmath,amsthm,amsfonts,amssymb,graphicx,epsfig,latexsym,color,mathrsfs}
\usepackage{epsf,enumerate}
\usepackage{fullpage}

\newtheorem{thm}{Theorem}[section]
\newtheorem{lemma}[thm]{Lemma}
\newtheorem{cor}[thm]{Corollary}

\newtheorem{defn}[thm]{Definition}
\newtheorem{question}[thm]{Question}
\newtheorem{obs}[thm]{Observation}
{}

\theoremstyle{remark}

\newtheorem*{definition*}{Definition}
\newtheorem*{remark*}{Remark}

\def\qed{\hfill \ifhmode\unskip\nobreak\fi\quad\ifmmode\Box\else$\Box$\fi\\ }

\renewcommand{\phi}{\varphi}

\title{Multiple DP-coloring of  planar graphs without 3-cycles and normally adjacent 4-cycles}
\author{Huan Zhou \thanks {Department of Mathematics, Zhejiang Normal University,  China.  E-mail: {\tt huanzhou@zjnu.edu.cn}.} 
	\and Xuding Zhu\thanks{Department of Mathematics, Zhejiang Normal University,  China.  E-mail: {\tt xdzhu@zjnu.edu.cn}. Grant Numbers: NSFC 11971438,12026248, U20A2068. }}
\begin{document}
	
	\maketitle
	\begin{abstract}
		The concept of DP-coloring of a graph is a generalization of list coloring introduced by Dvo\v{r}\'{a}k and Postle in 2015. Multiple DP-coloring of graphs, as a generalization of multiple list coloring, was first studied by   Bernshteyn, Kostochka and Zhu in 2019. This paper proves that planar graphs without 3-cycles and normally adjacent 4-cycles are $(7m, 2m)$-DP-colorable for every integer $m$. 	As a consequence, the strong fractional choice number of any planar graph  without 3-cycles and normally adjacent 4-cycles is at most $7/2$.	
		
	\end{abstract}

	\section{Introduction}
	
	A $b$-fold coloring of a graph $G$ is a mapping $\phi$ which assigns to each vertex $v$  a set $\phi (v)$ of $b$ colors so that adjacent vertices receive disjoint color sets. An $(a,b)$-coloring of $G$ is a $b$-fold coloring $\phi$ of $G$ such that $\phi(v) \subseteq \{1,2,\cdots,a\}$ for each vertex $v$. The {\em fractional chromatic number} of $G$ is
	$$\chi_f(G)=\inf\{\frac a b :G \ \text{ is $(a, b)$-colorable}\}.$$
	
	An \emph{ $a$-list assignment } of $G$ is a mapping $L$ which assigns to each vertex $v$ a set $L(v)$ of $a$ permissible colors. A \emph{$b$-fold $L$-coloring} of $G$ is a $b$-fold coloring $\phi$ of $G$ such that $\phi(v)\subseteq L(v)$ for each vertex $v$. We say $G$ is \emph{$(a, b)$-choosable} if for any $a$-list assignment $L$ of $G$, there is a $b$-fold $L$-coloring of $G$. The {\em choice number} of $G$ is 
	$$ch(G) = \min\{a: G \ \text{ is $(a,1)$-choosable.} \}.$$ The {\em fractional choice number} of $G$ is $$ch_f(G)=\inf\{r :G \ \text{ is $(a, b)$-choosable for some positive integers $a,b$ with $a/b = r$} \}.$$ The {\em strong fractional choice number} of $G$ is 
	$$ch^*_f(G)=\inf\{r :G \ \text{ is $(a, b)$-choosable for all positive integers $a,b$ with $a/b \ge r$} \}.$$
	
	It was proved by Alon, Tuza and Voigt \cite{Chf}  that for any finite graph $G$, $\chi_f(G)=ch_f(G)$ and moreover the infimum in the definition of $ch_f(G)$ is attained and hence can be replaced by minimum. So the fractional choice number $ch_f(G)$ of a graph is not a new invariant. On the other hand,  the concept of strong fractional choice number, introduced in \cite{Zhu2018}, was intended to be a refinement of $ch(G)$.
	It follows from the definition that  $ch_f^*(G) \ge ch(G)-1$.
	However, it remains an open question whether  $ch_f^*(G) \le ch(G)$. 
	
	For a family $\mathcal{G}$ of graphs, let 
	$$ch(\mathcal{G} ) = \max\{ch(G): G \in \mathcal{G}\}, ch_f(\mathcal{G} ) = \max\{ch_f(G): G \in \mathcal{G}\}, ch_f^*({\mathcal{G}})=  \sup\{ch_f^*(G): G \in \mathcal{G}\}.$$
	We denote by $\mathcal{P}$ the family of planar graphs, and by $\mathcal{P}_{\Delta}$ the family of triangle free planar graphs.
	It is known that $ch(\mathcal{P} ) =5$, $ch(\mathcal{P}_{\Delta} ) =4$, $ch_f(\mathcal{P} ) =4$ and $ch_f(\mathcal{P}_{\Delta} ) =3$. It is easy to see that 	$ch_f^*(\mathcal{P} )  \le 5$ and $ch^*_f(\mathcal{P}_{\Delta} ) \le 4$, and these are the best known upper bounds for $ch_f^*(\mathcal{P} ) $ and $ch^*_f(\mathcal{P}_{\Delta} )$, respectively.
	The  best known   lower bounds for 
	$ch_f^*(\mathcal{P} )  $ and $ch^*_f(\mathcal{P}_{\Delta} )$ are obtained in \cite{XZ2020} and \cite{JZ2019} respectively:
	$$ch_f^*(\mathcal{P} ) \ge 4+1/3, \ ch^*_f(\mathcal{P}_{\Delta} ) \ge 3+ \frac1{17}.$$
It would be interesting to find better upper or lower bounds for   $ch_f^*(\mathcal{P} )$ and $ch^*_f(\mathcal{P}_{\Delta} )$. In particular, the following questions remain open:

\begin{question}
	\label{q1}
	Is it true that every planar graph is $(9,2)$-choosable?
\end{question}
	
\begin{question}
	\label{q2}
		 Is it true that every triangle free planar graph is $(7,2)$-choosable?
\end{question}

It follows from the Four Color Theorem that every planar graph is $(4m,m)$-colorable for any positive integer $m$. However,  the problem of proving 
every planar graph is $(9,2)$-colorable without using the Four Color Theorem remained open for a long time, before it was done by Cranston and Rabern in 2018 \cite{CR2018}. As a weaker version of Question \ref{q1}, it was proved by Han, Kierstead and Zhu \cite{HKZ2021} that  every planar graph $G$ is $1$-defective $(9,2)$-paintable (and hence 1-defective $(9,2)$-choosable), where a 1-defective coloring is a coloring in which each vertex $v$ has at most one neighbour colored the same color as $v$. 
 
This paper studies a variation of Question \ref{q2}. We consider a   more restrictive family of graphs: the family of planar graphs without 3-cycle and without normally adjacent 4-cycles, where  
two 4-cycles are said to be {\em normally adjacent} if they share exactly one edge. We prove a stronger conclusion for this family of graphs, i.e., all graphs in this family are $(7m,2m)$-DP-colorable for all positive integer $m$.

The concept of DP-coloring is a generalization of list coloring   introduced by  Dvo\v{r}\'{a}k and Postle in \cite{DP}.  For $v \in V(G)$, $N_G(v)$ is the set of neighbours of $v$ and $N_G[v] = N_G(v) \cup \{v\}$.

\begin{defn}
	Let $G$ be a graph. A \emph{cover} of $G$ is a pair $(L, H)$, where $H$ is a graph and $L \colon V(G) \to Pow(V(H))$ is a function, with the following properties:
	\begin{itemize}
		\item The sets $\{L(u): u \in V(G)\}$  form a partition of $V(H)$.
		\item If $u, v \in V(G)$ and $L(v) \cap N_H(L(u)) \neq \emptyset$, then $v \in     N_G[u]$.
		\item Each of the graphs $H[L(u)]$, $u \in V(G)$, is complete.
		\item If $uv \in E(G)$, then $E_H(L(u), L(v))$ is a matching (not necessarily perfect and possibly empty).
	\end{itemize}
\end{defn}

We denote by $\mathbb{N}$ the set of non-negative integers. For a set $X$,   denote by $\mathbb{N}^X$ the set of mappings $f: X \to \mathbb{N}$. For a graph $G$, we write $\mathbb{N}^G$ for $\mathbb{N}^{V(G)}$. 

For $f,g \in \mathbb{N}^G$, we write $g \le f$ if $g(v) \le f(v)$ for each vertex $v$ of $G$, and let $(f+g) \in \mathbb{N}^G$ be defined as $(f+g)(v)=f(v)+g(v)$ for each vertex $v$ of $G$. If $G'$ is a subgraph of $G$, $f \in \mathbb{N}^G$, $g \in \mathbb{N}^{G'}$, we write $g \le f$ if $g(v) \le f(v)$ for each vertex $v$ of $G'$.

	For $f \in \mathbb{N}^G$, an $f$-cover of $G$ is a cover $(L, H)$ of $G$ with $|L(v)|=f(v)$ for each vertex $v$. 

\begin{defn}
	Let $G$ be a graph and let $(L, H)$ be a cover of $G$. An {\em $(L, H)$-coloring} of $G$ is an independent set $I$ of size $|V(G)|$. If for every $f$-cover $(L, H)$ of $G$, there is an $(L, H)$-coloring of $G$, then we say $G$ is {\em DP-$f$-colorable}. We say $G$ is {\em DP-$k$-colorable} if $G$ is DP-$f$-colorable for the constant mapping $f$ with $f(v)=k$ for all $v$. The {\em DP-chromatic number} of $G$ is defined as 
	$$\chi_{DP}(G)= \min \{k: \text{ \rm $G$ is DP-$k$-colorable} \}. $$ 
\end{defn}

List coloring of a graph $G$ is a special case of a DP-coloring of $G$: assume $L'$ is an $f$-list assignment of $G$, which assigns to each vertex $v$ a set $L'(v)$ of $f(v)$ permissible colors.   
Let $(L, H)$ be the $f$-cover graph of $G$ defined as follows:
\begin{itemize}
	\item For each vertex $v$ of $G$, $L(v) = \{v\} \times L'(v)$.
	\item For each edge $uv$ of $G$, connect $(v, c)$ and $(u, c')$ by an edge in $H$  if $c = c'$.
\end{itemize}

Then a mapping $\phi$ is an $L'$-coloring of $G$ if and only if the set $\{(v, \phi(v)): v \in V(G)\}$ is an independent set of $H$.   Therefore, for each graph $G$, $$ch(G) \le \chi_{DP}(G),$$
and it is known that the difference $\chi_{DP}(G)-ch(G)$ can be arbitrarily large.

Multiple DP-coloring of graphs was first studied in \cite{BKZ}. Given a cover $\mathcal{H} = (L, H)$ of a graph $G$, we refer to the edges of $H$ connecting distinct parts of the partition $\{L(v) : v \in V (G)\}$ as \em{cross-edges}. A subset $S \subset V (H)$ is \em{quasi-independent} if $H[S]$ contains no cross-edges.

\begin{defn}
	 	Assume $\mathcal{H}=(L,H)$ is a cover of $G$ and  $g \in \mathbb{N}^G$. An $(\mathcal{H}, g)$-coloring 
	 	is a quasi-independent set $S\subset V (H)$ such that $|S \cap L(v)| = g(v)$ for each $v\in V(G)$.  
	 	We say $G$ is   $(\mathcal{H}, g)$-colorable if there exists an $(\mathcal{H}, g)$-coloring of $G$. 
	 	We say graph $G$ is $(f,g)$-DP-colorable if for any $f$-cover $\mathcal{H}$ of $G$, $G$ is $(\mathcal{H}, g)$-colorable.  
	 	If $f,g\in \mathbb{N}^G$ are constant maps with $g(v)=b$   and $f(v)=a$ for all $v \in V(G)$, then $(\mathcal{H}, g)$-colorable is called $(\mathcal{H}, b)$-colorable, and
	 	$(f,g)$-DP-colorable is called $(a,b)$-DP-colorable.  
\end{defn} 
 
Similarly, we can show that $(a,b)$-DP-colorable implies $(a, b)$-choosable.

\begin{defn}
	 The fractional DP-chromatic number, $\chi_{DP}^{*}$, of $G$ is defined in \cite{BKZ} as
	 $$\chi_{DP}^{*}(G)= inf\{r: \text{\rm $G$ is $(a, b)$-DP-colorable for some $a/b = r$}\}.$$
	 We  define the {\em strong fractional DP-chromatic number} as 
	 $$\chi_{DP}^{**}(G)= inf\{r: \text{\rm $G$ is $(a, b)$-DP-colorable for every $a/b \ge r$}\}.$$
\end{defn}

\begin{obs}	
As $(a, b)$-DP-colorable implies $(a, b)$-choosable, we have  $$ch_f (G) \leq \chi_{DP}^{*} (G), ch_f^{*} (G) \leq \chi_{DP}^{**} (G).$$ 
It follows from the definition that $$\chi_{DP}^{*}(G) \le \chi_{DP}(G) \text{ and } \chi_{DP}^{**}(G) \ge \chi_{DP}(G)-1.$$ It was proved in   \cite{BKZ} that there are large girth graphs $G$ with $\chi(G)=d$ and $\chi_{DP}^*(G) \le d/ \log d$. As   $\chi_{DP}(G) \ge ch(G) \ge  \chi(G)$,   
the difference $\chi_{DP}^{**}(G) - \chi_{DP}^*(G)$  
can be arbitrarily large. 
\end{obs}

The following is the main result of this paper.  
\begin{thm}
		\label{thm-a}
	Let $G$ be a planar graph without $C_3$ and normally adjacent $C_4$. Then $G$ is $(7m, 2m)$-DP-colorable for every integer $m$.
	\end{thm}
	
As $(7m,2m)$-DP-colorable implies $(7m,2m)$-choosable, we have the following corollary.

\begin{cor}
\label{cor-strong fractional choice numer}
	If $G$ is a planar graph without $C_3$ and normally adjacent $C_4$, then $ch_f^*(G) \le 7/2$.
\end{cor}

The following  notations  will be used in the remainder of this paper. Assume $G$ is a graph.
A $k$-vertex ($k^+$-vertex, $k^-$-vertex, respectively) is a vertex of degree $k$ (at least $k$, at most $k$, respectively). A $k$-face, $k^-$-face or a $k^+$-face is a face of degree $k$, at most $k$ or at least $k$, respectively.  The notions of $k$-neighbor, $k^+$-neighbor, $k^-$-neighbor are
 defined similarly.  Two faces are \emph{intersecting} (respectively, \emph{adjacent} or \emph{normally adjacent}) if they share at least one vertex (respectively, at least one edge or exactly one edge). For a face $f\in F$, if the vertices on $f$ in a cyclic order are $v_1, v_2,\ldots, v_k$, then we write $f = [v_1v_2 \ldots v_k]$, and call $f$ a $(d(v_1), d(v_2),\ldots,d(v_k))$-face.

We use the following conventions in this paper:
\begin{enumerate}
	\item For any $f$-cover $\mathcal{H}=(L,H)$ of a graph $G$, for any edge $e=uv$ of $G$ with $f(u) \le f(v)$, we assume that the matching between $L(u)$ and $L(v)$ has $f(u)$ edges, and hence saturates $L(u)$, because adding edges to the matching only makes it more difficult to color the graph.  
	\item If the vertices of a graph $G$ is labelled as $v_1, v_2, \ldots, v_n$, then a mapping $f \in \mathbb{N}^G$ will be given as an integer sequence $(f(v_1), \ldots, f(v_n))$.
	\item  For an $f$-cover $\mathcal{H}=(L,H)$ of a graph $G$, an induced subgraph $H'$ of $H$ defines an $f'$-cover $\mathcal{H}'=(L',H')$ of $G$, where for each vertex $v$,  $L'(v) = L(v) \cap V(H') $ and $f'(v)=|L'(v)|$.  
\end{enumerate}

\section{Strongly extendable coloring of a subset}

Assume $G$ is a graph, $f,g \in \mathbb{N}^G$,  $X $ is a subset of $V(G)$,  $\mathcal{H}=(L,H)$ is an $f$-cover of $G$. By considering restriction of these mappings, we shall treat $\mathcal{H}$ as an $f$-cover of $G[X]$. Hence we can talk about  $(\mathcal{H}, g)$-coloring  of $G[X]$.

  Assume $G$ is a graph and  $X$ is a vertex cut-set. If   $G_1,G_2$ are induced subgraphs of $G$ such that 
$V(G_1) \cup V(G_2) = V(G)$ and $V(G_1) \cap V(G_2) = X$, then we say $G_1,G_2$ are the \emph{components} of $G$ separated by $X$. 

In an inductive proof, if every proper coloring of $X$ can be extended to a proper coloring of $G_2$, then we can first color $G_1$, and then extend it to $G_2$ to obtain a proper coloring of the whole graph. In our proofs below, usually $G_2$ do not have the property that every $(\mathcal{H}, g)$-coloring of $G[X]$ can be extended to an $(\mathcal{H}, g)$-coloring of $G_2$. Nevertheless, every $(\mathcal{H}, g)$-coloring $\phi$ of $G[X]$ satisfying the property that $\phi(v) \supseteq h(v)$ for some pre-chosen subsets $h(v)$ can be extended to an $(\mathcal{H}, g)$-coloring of $G_2$. In many cases, this property is   enough for the induction to be carried out. This technique is frequently used in the proofs below. We first give a precise definition of the desired property.

Assume $\phi$ is an $(\mathcal{H}, g)$-coloring  of $G[X]$ and $\phi'$ is an $(\mathcal{H}, g)$-coloring  of $G$.
If $\phi'(v)=\phi(v)$ for each vertex $v \in X$, then we say $\phi'$ is an extension of $\phi$. We say $\phi$ is {\em   $(\mathcal{H}, g)$-extendable}  if there exists  an $(\mathcal{H}, g)$-coloring  of $G$ which is an extension of $\phi$ to $G$.

 
\begin{defn}
	 Assume $G$ is a graph,  $f, h, h' \in \mathbb{N}^G$,   $h \le h' \le f$, $\mathcal{H}=(L,H)$ is an $f$-cover of $G$.
	 Assume   $\phi$ is an 
	  $(\mathcal{H}, h)$-coloring of $G $. An {\em $h'$-augmentation} of $\phi$ is an   $(\mathcal{H}, h')$-coloring $\phi'$ of $G $ such that
	   $\phi(v) \subseteq \phi'(v)$ for each vertex $v \in V(G)$.
\end{defn}

\begin{defn}
	Assume $G$ is a graph, $X$ is a subset  of $V(G)$,  $f,g,h \in \mathbb{N}^G$ and $h \le g \le f$. 
	Assume $\mathcal{H}=(L,H)$ is an $f$-cover of $G$. An 
		$(\mathcal{H}, h)$-coloring $\phi$ of $G[X]$ is called   {\em strongly   $(\mathcal{H}, g)$-extendable} if
		\begin{itemize}
			\item $\phi$ has an $g$-augmentation.
			\item Every 
			$g$-augmentation of $\phi$ is 
		 $(\mathcal{H}, g)$-extendable.
		\end{itemize}  
		We say  {\em $(f,h)$ is  strongly $(f,g)$ extendable from $X$ to $G$}, written as $$(f,h)_X \preceq (f,g)_G,$$  if for any $f$-cover $\mathcal{H}=(L,H)$ of $G$, there exists a strongly  $(\mathcal{H}, g)$-extendable $(\mathcal{H}, h)$-coloring  of $G[X]$. 
\end{defn}

The following lemma illustrates how the concept of strongly reducible coloring of an induced subgraph can be used to prove the $(f,g)$-DP-colorability of a graph.

\begin{lemma}
	\label{lem-BB}
	Assume $G$ is a graph, $X$ is a cut-set of $G$ and $G_1,G_2$ are components of $G$ separated by $X$.  Assume $f,g, h \in \mathbb{N}^G$ and $h \le g \le f$. Let $f', g' \in \mathbb{N}^G$ be defined as follows:
	\begin{enumerate}
		\item  $f'(v) = f(v) - \sum_{u \in N_G[v] \cap X}h(u)$ for $v \in V(G_2)$, and $f'(v)= f(v)$ for $v \notin V(G_2)$. 
		\item $g'(v)= g(v) - h(v)$ for $v \in X$, and $g'(v)=g(v)$ for $v \notin X$.  
	\end{enumerate}
	If $(f,h)_X \preceq (f,g)_{G_1}$ and $G_2$ is $(f',g')$-DP-colorable, then $G$ is $(f,g)$-DP-colorable. 
\end{lemma}
\begin{proof}
	Let $\mathcal{H}=(L,H)$ be an $f$-cover of $G$. Since $(f,h)_X \preceq (f,g)_{G_1}$, there exists an $(\mathcal{H}, h)$-coloring $\phi$ of $G[X]$, such that any $g$-augmentation $\phi'$ of $\phi$ can be extended to an $(\mathcal{H}, g)$-coloring  of $G_1$.
	
	Let $H'=H-N_H[\cup_{v \in X}\phi(v)]$. It is straightforward to verify that 
	$\mathcal{H}'=(L',H')$ is an $f'$-cover of $G_2$. Since $G_2$ is $(f',g')$-DP-colorable, there exists an $(\mathcal{H}', g')$-coloring $\psi$ of $G_2$. 
	
	For  $v \in X$, let 
	$\psi'(v) = \psi(v) \cup \phi(v)$. Then 
	$\psi'$, as a coloring of $G[X]$,  is a $g$-augmentation of $\phi$, and hence can be extended to an 
	$(\mathcal{H}, g)$-coloring of $G_1$, which we also denote by $\psi'$. Then $\psi''$ defined as 
	\[
	\psi''(v) = \begin{cases} \psi'(v), &\text{ if $v \in V(G_1)$}, \cr
	\psi(v), &\text{ if $v \notin V(G_1)$}\cr
	\end{cases}
	\]
is an $(\mathcal{H}, g)$-coloring of $G$. 
\end{proof}

Observe that as $\phi$ is an $(\mathcal{H}, h)$-coloring of $G[X]$, 
 a $g$-augmentation of $\phi$ is an  $(\mathcal{H}, g)$-coloring of $G[X]$.

In the formula 
$(f,h)_X \preceq (f, g)_G$, if $h$ or $g$ is a constant function, then we replace it by a constant. For example, we write  
$(f,b)_X \preceq (f, a)_G$ for $(f,h)_X \preceq (f, g)_G$ where $h(v)=b$ for $v \in X$ and $g(v)=a$ for $v \in V(G)$.

 Note that in the statement $(f,h)_X \preceq (f,g)_G$, the values of  $h(v)$ for $v \notin X$ are irrelevant. 
 
Given a partial $(\mathcal{H}, g)$-coloring $\phi$ of $G$, for each vertex $v$,   $\phi(v)$ is a subset of $L(v)$, and is treated as a subset of $V(H)$. For example, $H'=H - N_H(\phi(v))$ is a subgraph of $H$ and hence defines a cover $\mathcal{H}' = (L',H')$ of $G$.

\begin{lemma}
	\label{lemB}
		Assume $G$ is a graph, $X$ is a subset of $V(G)$,    $f,g, h,h' \in \mathbb{N}^G$ and  $h \le h' \le g \le f$.  Then  $$(f,h)_X \preceq (f,g)_G \Rightarrow (f,h')_X \preceq (f,g)_G.$$
		If $X'$ is a subset of $X$, then 
		$$(f,h)_X \preceq (f,g)_G \Rightarrow (f,h)_{X'} \preceq (f,g)_G.$$
\end{lemma}
\begin{proof}
Assume  $\mathcal{H}=(L,H)$ is an $f$-cover of $G$ and $\phi$ is a strongly  $(\mathcal{H}, g)$-extendable $(\mathcal{H}, h)$-coloring  of $G[X]$. Since $\phi$ has a $g$-augmentation, there is a $h'$-augmentation $\phi'$ of $\phi$.
As any $g$-augmentation of $\phi'$ extends to a $g$-augmentation of $\phi$, we conclude that every $g$-augmentation of $\phi'$ is  $(\mathcal{H}, g)$-extendable.  Hence  $(f,h')_X \preceq (f,g)_G$.

The second half of the lemma is proved similarly and is omitted.
\end{proof}

Note that   for any $h \le g \le f \in \mathbb{N}^G$, $X \subseteq  V(G)$, $$(f,h)_X \preceq (f,g)_G$$ implies that $G$ is $(f,g)$-DP-colorable, and 
$$(f,g)_X \preceq (f,g)_G$$
is equivalent to say that $G$ is $(f,g)$-DP-colorable.

\begin{lemma}
	\label{lemC}
		Assume $G$ is a graph, $X$ is a cut-set of $G$ and $G_1,G_2$ are components of $G$ separated by $X$.	 
 Assume $X_i \subseteq V(G_i)$, $X \subseteq X_i$,  $f,g, h_1,h_2 \in \mathbb{N}^G$, and for $i=1,2$, $h_i(v) =0$ for $v \notin X_i$. If $h_1+h_2 \le g$, then 
	 $$(f,h_1)_{X_1} \preceq (f,g)_{G_1} \text{ and } (f,h_2)_{X_2} \preceq (f,g)_{G_2}  \Rightarrow (f,h_1+h_2)_{X_1 \cup X_2} \preceq (f,g)_G.$$
\end{lemma}
 \begin{proof}
 	Assume  $\mathcal{H}=(L,H)$ is an $f$-cover of $G$ and for $i=1,2$, $\phi_i$ is an $(\mathcal{H}, h_i)$-coloring  of $  G[X_i]$ which is 
 	 strongly  $(\mathcal{H}, g)$-extendable to $G_i$.
 	Let $\phi' $ be the multiple coloring of $G[X_1 \cup X_2]$ defined as follows:
 	\[
 	\phi'(v) = \begin{cases} \phi_1(v) \cup \phi_2(v), &\text{ if $ v \in X$}, \cr
 	\phi_i(v), &\text{ if $v \in X_i-X_{3-i}$}. 
 	\end{cases}
 	\]
 	Note that $|\phi'(v)| \le (h_1+h_2)(v)$ for $v \in X$.
 	By arbitrarily adding some colors from $L(v)$ to $\phi'(v)$ if needed, we may assume that $|\phi'(v)| = (h_1+h_2)(v)$ for $v \in X$. 
 	Then $\phi'$ is an $(\mathcal{H}, h')$-coloring  of $G[X_1 \cup X_2]$. 
 	For any $g$-augmentation of $\phi'$, its restriction to $X_i$,  is a $g$-augmentation of $\phi_i$, and hence can be extended to an $(\mathcal{H}, g)$-coloring $\phi'_i$ of $G_i$. Note that 
 	$\phi'_1$ and $\phi'_2$ agree on the intersection $V(G_1) \cap V(G_2)=X$. Hence the union $\phi'_1 \cup \phi'_2$ is an $(\mathcal{H}, g)$-coloring of $G$. Therefore 
 	$$(f,h_1+h_2)_{X_1 \cup X_2} \preceq (f,g)_{G}.$$
 \end{proof}

\begin{lemma}
	\label{lem-key}
Assume $G$ is   a 3-path $v_{1}v_{2}v_{3}$, $X=\{v_1, v_3\}$, $f, g, h \in \mathbb{N}^G$, with $h=(p,0,p) \le g \le f$.
If
$$f(v_1)-f(v_2)+f(v_3) \ge p, f(v_2) \ge g(v_1)+g(v_2)+g(v_3) - p,$$
then $$(f,h)_X \preceq (f,g)_G.$$
			\end{lemma}
			\begin{proof} We prove the lemma  by induction on 
				$p$. If $p = 0$, then $f(v_2) \ge g(v_1)+g(v_2)+g(v_3)$ implies that any $(\mathcal{H}, g)$-coloring of $X$ can be extended to an $(\mathcal{H}, g)$-coloring of $G$. 
				
				Assume $ p > 0$. 	
					Assume $\mathcal{H} = (L, H)$ is an $f$-cover of $G$. 
					We consider two cases. 
				
		\noindent		
	{\bf Case 1} $f(v_1), f(v_3) \le f(v_2)$.
	
	Since 	$f(v_1)-f(v_2)+f(v_3) \ge h(v_1)$,   
	$ |L(v_2)  \cap N_H(L(v_1)) \cap N_H(L(v_3))| \ge  p$.
	
	Let   $U$ be a $p$-subset of $L(v_2)  \cap N_H(L(v_1)) \cap N_H(L(v_3))$, and for $i=1,3$, let 
	  $$\phi(v_i) = N_H(U) \cap L(v_i).$$
   Then $\phi$ is an $(\mathcal{H}, h)$-coloring of $G[X]$. 
   
   If $\phi'$ is a $g$-augmentation of $\phi$, then 
   $$|L(v_2) - (N_H(\phi'(v_1)) \cup \phi'(v_3))| \ge f(v_2) - p - (g(v_1)-p)-(g(v_3)-p) \ge g(v_2).$$
   We can extend $\phi'$ to an $(\mathcal{H},g)$-coloring of $G$ by letting $\phi'(v_2)$ be a $g(v_2)$-subset of $L(v_2) - (N_H(\phi'(v_1)) \cup \phi'(v_3))$. 
   So $\phi'$ is  $(\mathcal{H},g)$-extendable.

\noindent
{\bf Case 2}  $f(v_1) > f(v_2)$ or $f(v_3) > f(v_2)$. 

By symmetry, we may assume that 
  $f(v_1)-f(v_2) > 0$. Let 
  $$s = \min \{f(v_1)-f(v_2), p\}.$$
  
   Then there exists an $s$-element set $S$ of $L(v_1)$ such that $$S \cap N_H(L(v_2)) = \emptyset.$$
   
   We modify the mappings $f, g, h$ to $f',g', h'$ as follows:
   \begin{itemize}
   	\item $f'(v_i) = f(v_i)-s$  for $i=1,2,3$.
   	\item $h'(v_i)=h(v_i) -s$ and $g'(v_i)= g(v_i)-s$ for $i=1,3$, $g'(v_2)=g(v_2)$.
   \end{itemize}
   It is straightforward to verify that 
 $f',g', h'$ satisfy the condition of the lemma. So by induction hypothesis,
 $(f',h')_X \preceq (f',g')_G.$
 
 Let $T$ be an arbitrary $s$-subset of $L(v_3)$, and let  $T'$ be an $s$-subset of $L(v_2)$ which contains $N_H(T) \cap L(v_2)$.
 Let $H'=H- (S \cup T \cup T')$.  
 Then $\mathcal{H}'=(L',H')$ is
  an $f'$-cover of $G$. Let $\phi'$ be a strongly $X'$-$(\mathcal{H}', g')$-extendable $(\mathcal{H}', h')$-coloring of $G[X]$. 
  
  Let $$\phi(v_1) = \phi'(v_1) \cup S, \phi(v_3) = \phi'(v_3) \cup T.$$ We shall show that $\phi$ is a strongly  $(\mathcal{H}, g)$-extendable $(\mathcal{H}, h)$-coloring of $G[X]$.

  For any $g$-augmentation $\psi$ of $\phi$, 
  $$\psi'(v_1) = \psi(v_1)-S, \psi'(v_3)=\psi(v_3)-T$$ is a
  $g'$-augmentation of $\phi'$.  Hence $\psi'$ can be extended to an 
  $(\mathcal{H}', g')$-coloring $\psi^*$ of $G$. 
  Then $\phi^*=\psi^*$ except that $\phi^*(v_1) = \psi(v_1) \cup S$ and $\phi^*(v_3) = \psi^*(v_3) \cup T$ is an $(\mathcal{H}, g)$-coloring   of $G$ which is an extension of $\psi$.
			\end{proof}
	
The following corollary follows from Lemma  \ref{lem-BB} and Lemma \ref{lem-key}, and will be used frequently. 

 \begin{cor}
     	\label{cor-key}
     	Assume $G$ is a graph and  $v_1v_2v_3$ is an induced 3-path in $G$, $f, g \in \mathbb{N}^G$ and $k \le g(v_1), g(v_2)$ is a positive integer  such that $g \le f$ and    $f(v_1)+f(v_3) - f(v_2)  \ge k$. 
     	Let $f',g'  \in \mathbb{N}^G$ be defined as follows:
     	\begin{enumerate}
     		\item $f'(v_2) = f(v_2) - k$,  $g'(v_i) = g(v_i) - k$ for $i \in \{1,3\}$.
     		\item For $v \ne v_2$,   $f'(v) = f(v) - k |N_G[v] \cap \{v_1,v_3\}|$, and for $v \ne v_1, v_3$, $g'(v)=g(v)$. 
     	\end{enumerate}
     	If $G$ is $(f', g')$-DP-colorable, then $G$ is $(f,g)$-colorable.
     \end{cor}

			\begin{cor}
				\label{(3m,4m,3m)-path}
			Assume $G$ is a  3-path $v_1v_2v_3$.  
			\begin{enumerate}
				\item If $f=(3m,4m,3m)$,    then   $(f,2m)_{\{v_1,v_3\}} \preceq (f,2m)_G$. 
				\item If $f=(3m,5m,3m)$, then   $(f,m)_{\{v_1,v_3\}} \preceq (f,2m)_G$.
			\end{enumerate}
			\end{cor}

\section{$(f,2m)$-DP-colorable  graphs}	 

			\begin{lemma}
				\label{(3,4)-path}
				For $k\ge 1$, $G$ is a $k$-path $v_{1}v_{2}...v_{k}$, 
				$f \in \mathbb{N}^G$ such that 
				\begin{enumerate}
					\item $f(v_1) = f(v_k) = 3m$ and $f(v_{i})=3m$ or $5m$ for $i\in \{2, 3, \ldots, k-1\}$, 
					\item $f(v_{i})+f(v_{i+1}) \ge 8m$  for $i \in [k-1]$.
				\end{enumerate}
				Then   
				$$(f, m)_{\{v_1, v_k\}} \preceq (f,2m)_G.$$
				In particular, $G$ is $(f, 2m)$-DP-colorable.
			\end{lemma}
			\begin{proof}
			We prove this lemma by induction on $k$. If $k=1$, then 
			the lemma is obviously true. Assume $k \ge 2$ and the lemma holds for shorter paths. Since $f(v_1)+f(v_2)  \ge 8m$ and $f(v_1) = f(v_k) = 3m$, we know that $k \ge 3$. 
			If $k=3$, then this is Corollary \ref{(3m,4m,3m)-path}.
			Assume $k\ge 4$. 
			
			If $f(v_i) = 3m$ for some $3 \le i \le k-2$, then 
			let $G_1$ be the path $v_1\ldots v_i$ and $G_2$ be the path $v_i\ldots v_k$. By induction hypothesis,
			
			$$(f, m)_{\{v_1,v_i\}} \preceq (f,2m)_{G_1}, \text{ and } (f, m)_{\{v_i,v_k\}} \preceq (f,2m)_{G_2}.$$
			
		By letting $X=\{v_1, v_i, v_k\}$ and $h(v_1)=h(v_k)=m$ and
			$h(v_i)=2m$, it follows from Lemma \ref{lemC} that  
		 $(f, h)_X \preceq (f,2m)_G$, which is equivalent to 
			$(f, m)_{\{v_1, v_k\}} \preceq (f,2m)_G$.
			
			Assume $f(v_i) =5m$ for $i=2, \ldots, k-1$ and $k \ge 4$.
			In this case, we show a stronger result:   for   $h(v_1)=m$ and $h(v_k)=0$, $(f, h)_{\{v_1, v_k\}} \preceq (f,2m)_G$.
			
			Assume $\mathcal{H}=(L,H)$ is an $f$-cover of $G$. 
			We need to show that there exists an $m$-subset $S$ of $L(v_1)$ such that 
			for any $2m$-subset $S'$ of $L(v_1)$ containing $S$, and any $2m$-subset $T$ of $L(v_k)$, there exists an $(\mathcal{H}, 2m)$-coloring $\psi$ of $G$ such that 
			$\psi(v_1)=S'$ and $\psi(v_k)=T$. 
			
			Let $\mathcal{H}'$ be the restriction of $\mathcal{H}$ to $G-v_k$, except that $L'(v_{k-1})= L(v_{k-1}) - N_H(T)$. 
			Let $f'$ be the restriction of $f$ to $G-v_k$, except that $f'(v_{k-1}) = 3m$. Then $\mathcal{H}'$ is an $f'$-cover of $G-v_k$. By induction hypothesis, 
			$(f', m)_{\{v_1, v_{k-1}\}} \preceq (f', 2m)_{G-v_k}$. 
			Hence there exists an $m$-subset $S$ of $L(v_1)$ such that such that 
			for any $2m$-subset $S'$ of $L(v_1)$ containing $S$,   there exists an $(\mathcal{H}', 2m)$-coloring $\psi$ of $G-v_k$. Now $\psi$ extends to an $(\mathcal{H}, 2m)$-coloring $\psi'$ of $G$ with $\psi'(v_k) = T$.
			\end{proof}

			\begin{lemma}
				\label{(3,4)-cycle}
				Assume $G$ is a cycle $v_{1}v_{2}...v_{k}v_1$ such that $k \ge 4$,
				\begin{enumerate}
					\item $f(v_{i})=3m$ or $5m$ for $i\in[k]$,
					\item $f(v_i)+f(v_{i+1}) \ge 8m$ for $i \in [k]$.
				\end{enumerate}
				Then $G$ is $(f, 2m)$-DP-colorable.
			\end{lemma}
			\begin{proof}
				If there are two vertices $v_i$ and $v_j$ with 
				$f(v_i) = f(v_j) = 3m$, then let $P_1 = v_iv_{i+1}\ldots v_j$ and $P_2 = v_j v_{j+1} \ldots v_i$ be the two paths of $G$ connecting $v_i$ and $v_j$.  By Lemma \ref{(3,4)-path}, 
				$$(f,m)_{\{v_i,v_j\}} \preceq (f,2m)_{P_1}, \text{ and } (f,m)_{\{v_i,v_j\}} \preceq (f,2m)_{P_2}.$$
				It follows from Lemma \ref{lemB} that 
				$(f,2m)_{\{v_i,v_j\}} \preceq (f,2m)_G$.
				So $G$ is $(f, 2m)$-DP-colorable.
				
				Otherwise, we may assume that $f(v_i) = 5m$ for $i=2,3,\ldots, k$. Let $f'=f$ except that $f'(v_1) = f'(v_3) = 3m$. Then $f'$ satisfies the condition of the lemma, and by the previous paragraph, $G$ is $(f', 2m)$-DP-colorable, which implies that $G$ is $(f, 2m)$-DP-colorable. 
			\end{proof}

			\begin{lemma}
				\label{lem-4-vertex}
				Assume $G=K_{1,3}$ is star with  $v_4$ be the center and   $\{v_1, v_2, v_3\}$ be the three leaves. Then for  $f=(3m,3m,3m,5m)$, $G$ is $(f, 2m)$-DP-colorable.
			\end{lemma}
			\begin{proof}
				Apply Lemma \ref{lem-BB} to $(f,g)$ and  $(v_1,v_4,v_2)$, it suffices to show that $G$ is  $(f_1, g_1)$-DP-colorable, where $f_1=(2m,2m,3m,4m),g_1= (m,m,2m,2m)$.
				
				 Apply Lemma \ref{lem-BB}   to $(f_1,g_1)$ and $(v_2,v_4, v_3)$, it suffices to show that 
				 $G$ is $(f_2,g_2)$-DP-colorable, where $f_2=(2m,m,2m,3m), g_2=(m,0,m,2m)$.
				 (Now $v_2$ needs no more colors and can be deleted. However, to keep the labeling of the vertices, we do not delete it).
				 
				 Apply Lemma \ref{lem-BB}   to $(f_2,g_2)$ and $(v_1, v_4, v_3)$, it suffices to show that 
				 $G$ is $(f_3, g_3)$-DP-colorable, where $f_3=(m,m,m,2m), g_3=(0,0,0,2m)$, and this is obviously true. 
			\end{proof}

			\begin{lemma}
				\label{4star}
				Assume $G=K_{1,4}$ is a star with center $v_5$ and four leaves $v_1, v_2, v_3, v_4$. Let $f=(2m,2m,2m,2m, 4m), g=(m,m,m,m,2m)$.  Then $G$ is $(f,g)$-DP-colorable. 
			\end{lemma}	
			\begin{proof}
				Assume $\mathcal{H}=(L,H)$ is an $f$-cover of $G$.  We construct an $(\mathcal{H}, g)$-coloring $\phi$ of $G$ as follows:
				
				Initially let $\phi(v)=\emptyset$ for all $v \in V(G)$. 
				
				Assume $|N_H(L(v_1)) \cap N_H(L(v_2)) \cap L(v_5)| = a$. Let $k=\min\{a, m\}$, let 
				 $S_1(v_5)$ be a $k$-subset of $ N_H(L(v_1)) \cap N_H(L(v_2)) \cap L(v_5)$.

				For $i=1,2$, add  $  L(v_i) \cap N_H(S_1(v_5))$ to    $\phi(v_i)$. 
				Let $$H_1 = H - N_H[\phi(v_1) \cup \phi(v_2)], \ \text{ and } \mathcal{H}_1=(L_1, H_1).$$ 
			  Let $g_1(v_i)=g_1(v_i) - k$ for $i=1,2$, and $g_1(v_j)=g_1(v_j)$ for $j \ne 1,2$. 
				
				It suffices to show that  there exists an $(\mathcal{H}_1,g_1)$-coloring of $G$. If $k=m$, then $g_1(v_i)=0$ for $i=1,2$. So we can delete $v_1, v_2$. As $|L_1(v_5)| = 3m$, it follows from Lemma \ref{lem-key}   that  there exists an $(\mathcal{H}_1,g_1)$-coloring of $G$.
				
				Assume $k=a < m$. Then $N_H(L_1(v_1)) \cap N_H(L_1(v_2)) = \emptyset$. As $|L_1(v_5)| = 4m-k$ and $|L_1(v_3)|=|L_1(v_4)|=2m$, 
				we have   $$|L_1(v_5) \cap N_{H_1}(L_1(v_3))) \cap N_{H_1}(L_1(v_4)))| \ge k.$$ 
				
				Let $S_2(v_5)$ be a $k$-subset of $L_1(v_5) \cap N_{H_1}(L_1(v_3))) \cap N_{H_1}(L_1(v_4)))$. 	For $i=3,4$,  add $  L_1(v_i) \cap N_{H_1}(S_2(v_5))$  to $\phi(v_i)$. Let $$H_2 = H_1 - N_{H_1}[\phi(v_3) \cup \phi(v_4)], \text{ and } \mathcal{H}_2=(L_2, H_2).$$
				Let $g_2(v_i)=g_1(v_i) - k$ for $i=3,4$, and $g_2(v_j)=g_1(v_j)$ for $j \ne 3,4$. It suffices to show that  there exists an $(\mathcal{H}_2,g_2)$-coloring of $G$.
				
				As $N_{H_2}(L_2(v_1)) \cap N_{H_2}(L_2(v_2)) = \emptyset$, we conclude that $|N_{H_2}(L_2(v_1)) \cap N_{H_2}(L_2(v_3)) \cap L_2(v_5)| \ge m-k$, or $|N_{H_2}(L_2(v_2)) \cap N_{H_2}(L_2(v_3)) \cap L_2(v_5) | \ge m-k$. By symmetry, we assume  that 
				$$|N_{H_2}(L_2(v_1)) \cap N_{H_2}(L_2(v_3))\cap L_2(v_5) | \ge m-k.$$ 
					Let $S_3(v_5)$ be an $(m-k)$-subset of $L_2(v_5) \cap N_{H_2}(L_2(v_3))) \cap N_{H_2}(L_2(v_4)))$. 	For $i=3,4$,  add $  L_2(v_i) \cap N_{H_2}(S_3(v_5))$  to $\phi(v_i)$. Let $$H_3 = H_2 - N_{H_2}[\phi(v_3) \cup \phi(v_4)], \text{ and } \mathcal{H}_3=(L_3, H_3).$$
					Let $g_3(v_i)=g_2(v_i) - (m-k)$ for $i=1,3$, and $g_3(v_j)=g_2(v_j)$ for $j \ne 1, 3$. It suffices to show that  there exists an $(\mathcal{H}_3,g_3)$-coloring of $G$.
					
				Observe that $g_3(v_1)=g_3(v_3)=0$, and hence $v_1, v_3$ can be deleted. The remaining graph is a 3-path. It is easy to verify that $|L_3(v_5)| = 3m-k$ and $|L_3(v_2)| = |L_3(v_4)|=2m-k$, $g(v_5)=2m$ and $g_3(v_2)=g_3(v_4) = m-k$. It follows from Lemma \ref{lem-key} that $G$ is $(\mathcal{H}_3,g_3)$-colorable. 	
			\end{proof}
		
			\begin{figure}[htb]
				\centering
				\includegraphics[height=4.8cm,width=6cm]{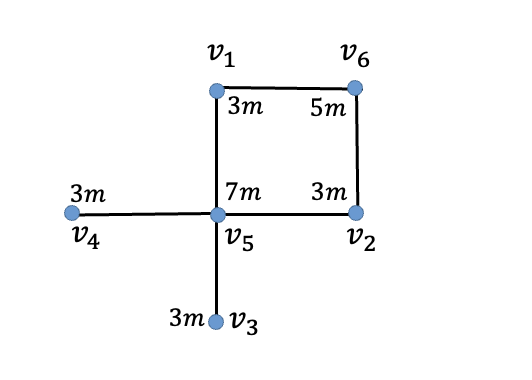}
				\caption{The graph $G$ and $f \in \mathbb{N}^G$}\label{4c+2}
			\end{figure}

			\begin{cor}
				\label{cor-4c+2}
				For the graph $G$ and $f \in \mathbb{N}^G$ shown in Figure \ref{4c+2}, $G$ is $(f,2m)$-DP-colorable. 
			\end{cor}
				\begin{proof}
					Let $G_1$ be the 3-path induced by $\{v_1,v_6,v_2\}$. By  Corollary \ref{(3m,4m,3m)-path},   
					$(f,m)_{\{v_1, v_2\}} \preceq (f,2m)_{G_1}$.
					
					Apply Lemma \ref{lem-BB} to the cut-set $X=\{v_1, v_2\}$, it suffices to show that   $G'=G[\{v_1, v_2, v_3, v_4, v_5\}]$ is $(f,g)$-DP-colorable, where $f=(2m,2m,3m,3m,5m)$ and $g=(m,m,2m,2m,2m)$. 
					
					Apply Corollary \ref{cor-key} to the 3-path $v_3v_5v_4$ with $k=m$,
					it suffices to show that $G'$ is $(f_1,g_1)$-DP-colorable, where $f_1=(2m,2m,2m,2m,4m)$ and $g_1=(m,m,m,m,2m)$. This follows from Lemma  \ref{4star}.		
				\end{proof}
			
			\begin{figure}[htb]
				\centering
				\includegraphics[height=4.8cm,width=6.5cm]{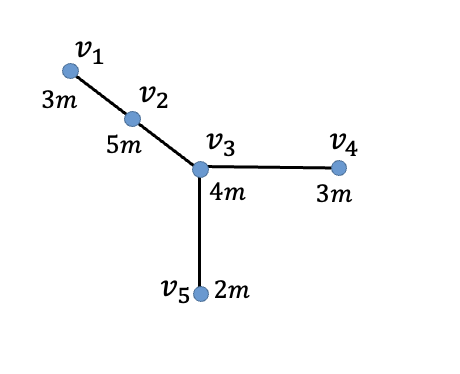}
				\caption{The graph $G$ and $f,g \in \mathbb{N}^G$}\label{4p+1}
			\end{figure}
			
	\begin{lemma}
				\label{lem-4p+1}
				For the graph $G$ and $f \in \mathbb{N}^G$ shown in Figure \ref{4p+1}. Let   $g=(2m,2m,2m,2m,m)$. Then $G$ is $(f,g)$-DP-colorable. 
			\end{lemma}
				\begin{proof}
					Apply Corollary \ref{cor-key} to the 3-path $v_4v_3v_5$ with $k=m$, it suffices to show that   $G'=G[\{v_1, v_2, v_3, v_4\}]$ is $(f',g')$-DP-colorable, where $f'=(3m,5m,3m,2m)$ and $g'=(2m,2m,2m,m)$. 

					Let $G_1$ be 3-path $v_1v_2v_3$ and $G_2$ be single edge $v_3v_4$. Apply Lemma \ref{lem-key} to $G_1$ with $p=m$ and  Lemma \ref{lem-BB}, it suffices to show that $G_2$ is $(2m,m)$-DP-colorable, which is obviously true. 	
				\end{proof}

\begin{figure}[htb]
				\centering
				\includegraphics[height=5.5cm,width=12cm]{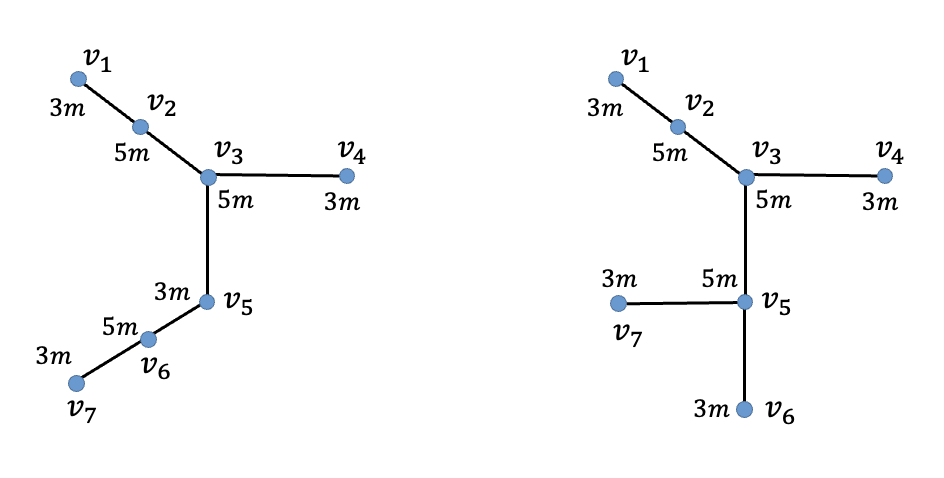}
				\caption{The graphs $G$ and $f \in \mathbb{N}^G$}\label{4p+3p}
			\end{figure}
			
			\begin{cor}
				\label{cor-4p+3p}
				For the graphs $G$ and $f \in \mathbb{N}^G$ shown in Figure \ref{4p+3p}, $G$ is $(f,2m)$-DP-colorable. 
	\end{cor}
	\begin{proof}
	First we show the left graph in Figure \ref{4p+3p} is $(f,2m)$-DP-colorable. Let $G_1$ be the 3-path induced by $\{v_5,v_6,v_7\}$. By Corollary \ref{(3m,4m,3m)-path},  $(f,m)_{\{v_5, v_7\}} \preceq (f,2m)_{G_1}$. Apply Lemma \ref{lem-BB} to the cut-set $X=\{v_5\}$, it suffices to show that  $G'=G[\{v_1, v_2, v_3, v_4, v_5\}]$ is $(f',g')$-DP-colorable, where $f'=(3m,5m,4m,3m,2m)$ and $g=(2m,2m,2m,2m,m)$. This follows from Lemma \ref{lem-4p+1}.
					
	Next we consider the right graph in Figure \ref{4p+3p}. Assume $\mathcal{H}=(L,H)$ is an $f$-cover of $G$.  We construct an $(\mathcal{H}, g)$-coloring $\phi$ of $G$ as follows:				 
	Let $S_1(v_5)$ be an $m$-subset of $L(v_5)- N_H(L(v_6))$, and add $S_1(v_5)$ to  $\phi(v_5)$. Choose a $2m$-subset from $L(v_7)-N_H(S_1(v_5))$ and add it to $\phi(v_7)$. It suffices to prove $G'=G[\{v_1, v_2, v_3, v_4, v_5,v_6\}]$ has an $(f',g')$-DP-coloring, where $f'=(3m,5m,4m,3m,2m,3m)$ and $g'=(2m,2m,2m,2m,m,2m)$. By Lemma \ref{lem-4p+1}, $G'-v_6$ has an $(f',g')$-DP-coloring $\phi'$. Choose a $2m$-subset   of $L(v_6)-\phi'(v_5)$ and add the $2m$-subset to $\phi(v_6)$. Let $\phi(v_i)=\phi'(v_i)$ for $i=1,2,3,4$ and $\phi(v_5)=\phi'(v_5)\cup S_1(v_5)$. Thus $\phi$ is an $(\mathcal{H}, g)$-coloring of $G$.
				 \end{proof}
	
	 \begin{figure}[htb]
			\centering
			\includegraphics[height=5.5cm,width=13cm]{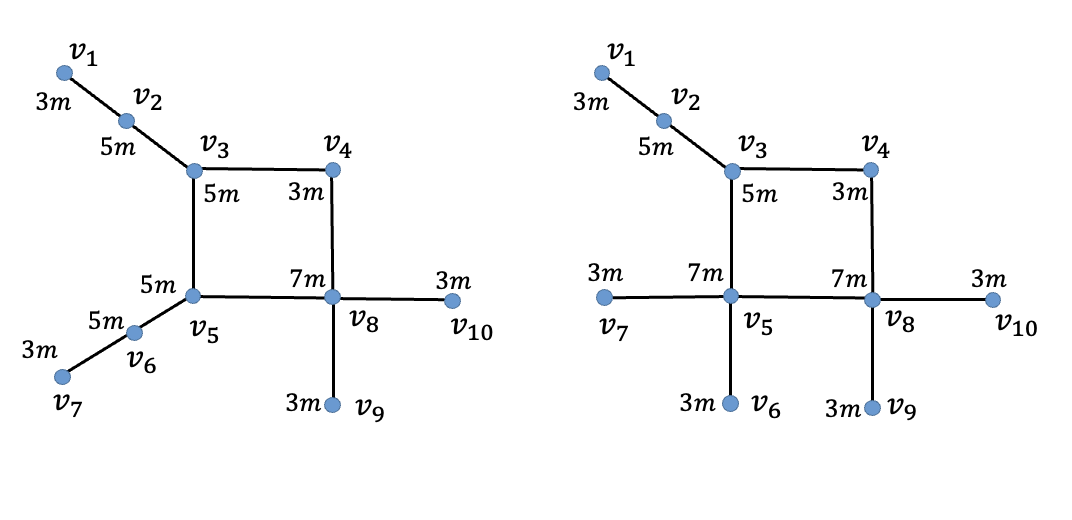}
			\caption{The graphs $G$ and $f \in \mathbb{N}^G$}\label{4p+3p+3p}
		\end{figure}
	\begin{cor}
				\label{cor-4p+3p+3p}
				For the graphs $G$ and $f \in \mathbb{N}^G$ shown in Figure \ref{4p+3p+3p}, $G$ is $(f,2m)$-DP-colorable. 
			\end{cor}
				\begin{proof}
				Assume $G$ is any of the two graphs in  Figure \ref{4p+3p+3p}, and $\mathcal{H}=(L,H)$ is an $f$-cover of $G$. Let $H' = H - L(v_8)\cap N_H(L(v_4))$ and $\mathcal{H}'=(L',H')$. 
				Let  $e=v_4v_8$. Then it suffices to show that $G'=G-e$ is    $(\mathcal{H}', 2m)$-colorable.
				
				By Corollary \ref{(3m,4m,3m)-path}, the subgraph $G'[v_8,v_9,v_{10}]$ has an  $(\mathcal{H}', 2m)$-coloring $\phi_1$. 
				 
				Let $H''=H'- L'(v_5) \cap N_{H'}(\phi_1(v_8))$. 
				 It remains  to prove that $G''=G[\{v_1, v_2, v_3, v_4, v_5,v_6,v_7\}]$ is $(\mathcal{H}'', 2m)$-coloring. For the graph $G$ on the left, 
				   $\mathcal{H}''$ is an $f'$-cover of $G''$, where
				   $f'=(3m,5m,5m,3m,3m,5m,3m)$.
				   For the graph $G$ on the right, 
				   $\mathcal{H}''$ is an $f'$-cover of $G''$, where
				   $f'=(3m,5m,5m,3m,5m,3m,3m)$. Now 
				    the conclusion  follows from Corollary \ref{cor-4p+3p}.
	 \end{proof}

\section{Proof of Theorem \ref{thm-a}}
	
Let $G$ be a counterexample to Theorem \ref{thm-a} with minimum number of vertices.
It is trivial that $G$ is connected and has minimum degree at least $3$.
Let $\mathcal{H}=(L, H)$ be a $7m$-cover of $G$ such that $G$ is not $(\mathcal{H}, 2m)$-colorable. By our assumption,  $E_{H}(L(u), L(v))$ is a perfect matching whenever $uv \in E(G)$.
		
In the following, for an induced subgraph $G'$ of $G$, we denote by $f' \in \mathbb{N}^{G'}$ the mapping  defined as  $f'(v) \ge 7m- 2(d_G(v)-d_{G'}(v))m$ for $v \in V(G')$.
		
	\begin{defn}
		A {\em configuration} in  $G$ is an induced subgraph $G'$ of $G$,  where each vertex $v$ of $G'$ is labelled with its degree $d_G(v)$ in $G$. A  configuration $G'$ is {\em reducible} if $G'$ is $(f', 2m)$-DP-colorable. 
	\end{defn}

	\begin{lemma}
		\label{reducibility}
		$G$ contains no reducible configuration. 
	\end{lemma}
	\begin{proof}
		Assume $G'$ is a reducible configuration in $G$. 
		By minimality of $G$, $G-G'$ has an $(\mathcal{H}, 2m)$-coloring $\phi$. For $v \in V(G')$, let $$L'(v) = L(v) - \cup_{u \in N_G(v) - V(G')} \phi(u)$$
		and $H' = H[ \cup_{v\in V(G')} L'(v)]$.
		Then $\mathcal{H}' = (L',H')$ is an $f'$-cover of $G'$.  As  $G'$ is reducible, $G'$ has an $(\mathcal{H}', 2m)$-coloring $\phi'$.
		Then $\phi \cup \phi'$ is an $(\mathcal{H}, 2m)$-coloring of $G$, a contradiction.
	\end{proof}

	\begin{cor}
\label{cor-reducible}
The following configurations in Figure \ref{reducible} are reducible.
\begin{figure}[htb]
	\centering
	\includegraphics[height=8cm,width=12cm]{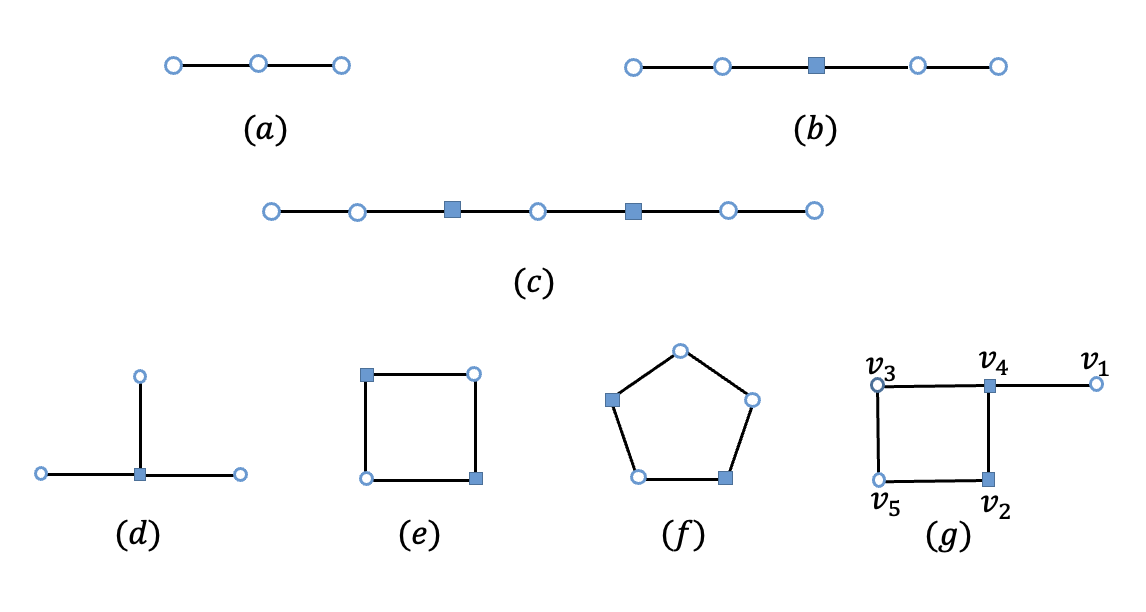}
	\caption{Reducible configurations, where  hollow circles is a 3-vertex, and   squares is a 4-vertex.	}\label{reducible}
		\end{figure}
	\end{cor}
	\begin{proof}
		The reducibility of configurations $(a),(b),(c)$  follows from Lemma \ref{(3,4)-path}, (d) follows from   Lemma \ref{lem-4-vertex},   $(e)$ and $(f)$ follows from Lemma \ref{(3,4)-cycle}.
		
		Now we prove the reducibility of configurations $(g)$. Let $G'=G[\{v_1,v_2,v_3,v_4,v_5\}]$. Let $f'(v) = 7m - 2(d_G(v)-d_{G'}(v))m$. Then $f'(v_i)=3m$ for $i=1,2$ and $f'(v_j)= 5m$ for $j=3,4,5$. Assume $\mathcal{H'}=(L',H')$ is an $f'$-cover of $G'$.  We color $v_5$ with a $2m$-subset $\phi(v_5)$ of $L'(v_5)-N_{H'}(L'(v_2))$. Let $\mathcal{H''} = \mathcal{H'}-L'(v_3)\cap N_{H'}(\phi(v_5))$. It suffices to prove $G''=G[\{v_1, v_2, v_3, v_4\}]$ has an $(\mathcal{H''}, 2m)$-coloring.
		As $\mathcal{H''}$ is an $f''$-cover,    where $f''=(3m,3m,3m,5m)$, this follows from Lemma \ref{lem-4-vertex}. 
	\end{proof}

	\begin{lemma}
		\label{4-faces with two 3-vertices}
		If two $4$-faces   intersect at a $4$-vertex,   then one of them contains at most one $3$-vertex.  
	\end{lemma}
	\begin{proof}
		Assume that $f_1$ and $f_2$ are $4$-faces intersect at a 4-vertex $v$, and each of $f_1,f_2$ contains at least two $3$-vertices.  Then either   $v$ is adjacent to three $3$-vertices and hence $G$ contains reducible configuration (d), or $G$ contains a  $(3, 3, 4, 3, 3)$-path, which is the reducible configuration (b).    
	\end{proof}
	
	We call a $4$-face $f$   {\em light  } if $f$ is $(4, 4, 3, 3)$-face, a $(4, 5, 3, 3)$-face or a $ (4, 3, 5, 3)$-face. (Note that $G$ contains no  $(4,3,4,3)$-face, as it is reducible by Corollary \ref{cor-reducible} (e)).
	
	Assume $v$ is a $4$-vertex. We say $v$ is 
	\begin{enumerate}
		\item   {\em strong} if it is not incident to any light $4$-face.
		\item   {\em{normal}} if it is incident to a light $4$-face and three $5^+$-faces. 
		\item   {\em{weak}} if it is incident to a light $4$-face and a $4$-face with no $3$-vertex.
		\item   {\em{very weak}} if it is incident to a light $4$-face and a $4$-face with a $3$-vertex. 
	\end{enumerate}

	Let $v$ be a weak or very weak $4$-vertex. If $v$ has a $3$-neighbor $u$ such that $vu$ is shared by a light $4$-face and a $5$-face $f$, then $f$  is called a {\em special $5$-face } of $v$.

\begin{lemma}
	\label{(4,4,3,3)-face} 	
A $(4, 4, 4, 3)$-face does not intersect a $(4, 4, 3, 3)$-face at a 4-vertex.
\begin{figure}[htb]
			\centering
			\includegraphics[height=5cm,width=11cm]{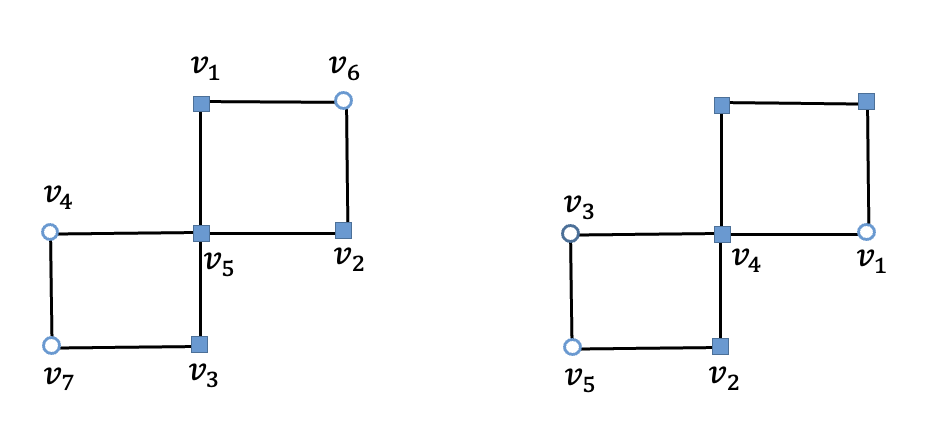}
			\caption{(4, 4, 4, 3)-face intersects (4, 4, 3, 3)-face}\label{fig1}
		\end{figure}	
	\end{lemma}
	
	\begin{proof}	
	Assume that a $(4, 4, 3, 3)$-face intersects a $(4, 4, 4, 3)$-face at a 4-vertex $v$. Thus one of the graphs in Figure \ref{fig1} is a  subgraph of $G$. Assume  $G'$ on the left  of Fig. \ref{fig1} is a  subgraph of $G$.  Since $G$ is triangle free, contains no $(3,3,3)$-path and   no normally adjacent $4$-cycles, $G'$ is an induced subgraph of $G$. We shall prove that $G'$ is reducible. 
	
	Note that $f'=(3m,3m,3m,5m,7m, 5m,5m)$. 
	 Assume $\mathcal{H'}=(L',H')$ is an $f'$-cover of $G'$.  We color $v_7$ with a $2m$-subset $\phi(v_7)$ of    $L'(v_7)-N_{H'}(L'(v_3))$. Let $\mathcal{H''} = \mathcal{H'}-L'(v_4)\cap N_{H'}(\phi(v_7))$. It suffices to prove $G''=G[\{v_1, v_2, v_3, v_4, v_5,v_6\}]$ has an $(\mathcal{H''}, 2m)$-coloring. As $\mathcal{H}''$ is an $f''$-cover of $G''$, where  $f''=(3m,3m,3m,3m,7m,5m)$, the result follows from Corollary \ref{cor-4c+2}. Thus $G'$ is reducible, a contradiction. 
	
	Assume the graph on the right of Figure \ref{fig1} is a subgraph of $G$. Then $G'=G[\{v_1,v_2,v_3,v_4,v_5\}]$ is the reducible configuration (g), a contradiction. 
	\end{proof}

\begin{lemma}
	\label{(4,3,5,3)-face} 
A $(4, 4, 4, 3)$-face does not intersect a $(4, 3, 5, 3)$-face at a 4-vertex.
\begin{figure}[htb]
\centering
\includegraphics[height=5cm,width=6cm]{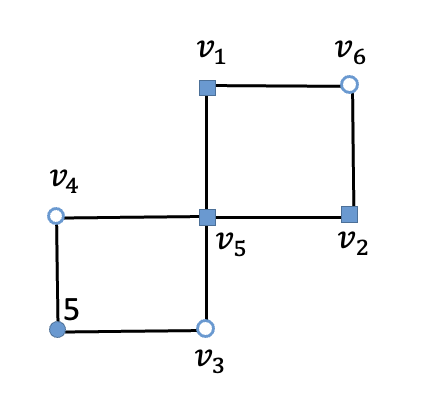}
\caption{(4, 3, 5, 3)-face intersects (4, 4, 4, 3)-face}\label{fig2}
\end{figure}	
\end{lemma}
	
	\begin{proof}
	Assume a $(4, 3, 5, 3)$-face $f_1$ intersect a $(4, 4, 4, 3)$-face $f_2$ at a 4-vertex. By Corollary \ref{cor-reducible}(d), a $4$-vertex has at most two $3$-neighbors. Thus the $4$-cycles are as shown in Figure \ref{fig2}. But the  induced subgraph $G'=G[\{v_1,v_2,v_3,v_4,v_5,v_6\}]$ is reducible by Corollary \ref{cor-4c+2}, a contradiction. 
		\end{proof}

\begin{lemma}
	\label{(4,4,4,3)-face}
 A $(4^+, 4^+, 4^+, 3)$-face  contains at most one very weak $4$-vertex.
\end{lemma}

\begin{proof}
Assume that $f =(v_1,v_2,v_3,v_4)$ is a $(4^+, 4^+, 4^+, 3)$-face and contains two very weak $4$-vertices. 

If $v_1$ and $v_3$ are very weak 4-vertices, then since a $4$-vertex has at most two $3$-neighbors, the light faces incident to $v_1$ and $v_3$ are $(4, 4^+, 3, 3)$-faces. This implies that $G$ has a $(3, 3, 4, 3, 4, 3, 3)$-path in $G$, which is a reducible configuration (c), a contradiction. 

Thus we assume that $v_1, v_2$ are very weak $4$-vertices. Using the fact that a $4$-vertex has at most two $3$-neighbors, we conclude that $G$ contains one of the  graphs in Figure \ref{fig3} as an induced subgraph. But   by Corollary \ref{cor-4p+3p}, the subgraph $G[v_1,v_2,v_4,v_5,v_6, v_7,v_8]$ is reducible, a contradiction. 
 
 \begin{figure}[htb]
			\centering
			\includegraphics[height=6cm,width=12cm]{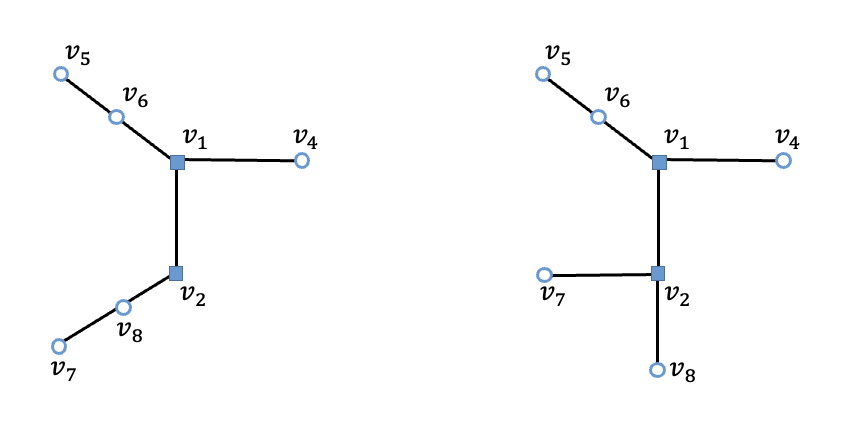}
			\caption{$(4, 4, 4^+, 3)$-face with two very weak 4-vertices}\label{fig3}
		\end{figure}

\end{proof}

\begin{lemma}
	\label{(4,4,4,4)-face}
 Assume a $(4, 4, 4, 4)$-face $f$ contains a weak $4$-vertex, which is    incident to a $(4, 3, 5, 3)$-face. Then $f$ contains at most two weak $4$-vertices. 
 
               \begin{figure}[h]
\centering
\includegraphics[height=4.7cm,width=6cm]{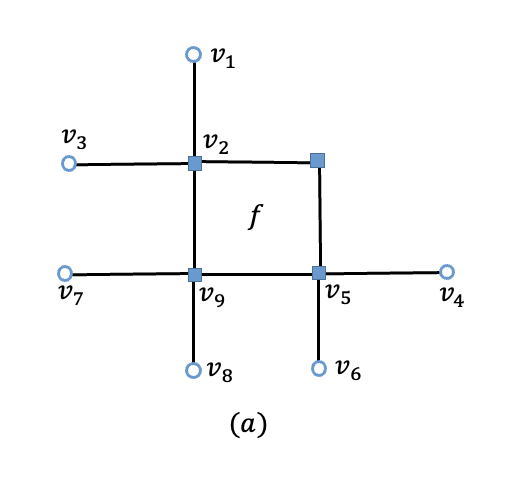}
\includegraphics[height=4.7cm,width=12cm]{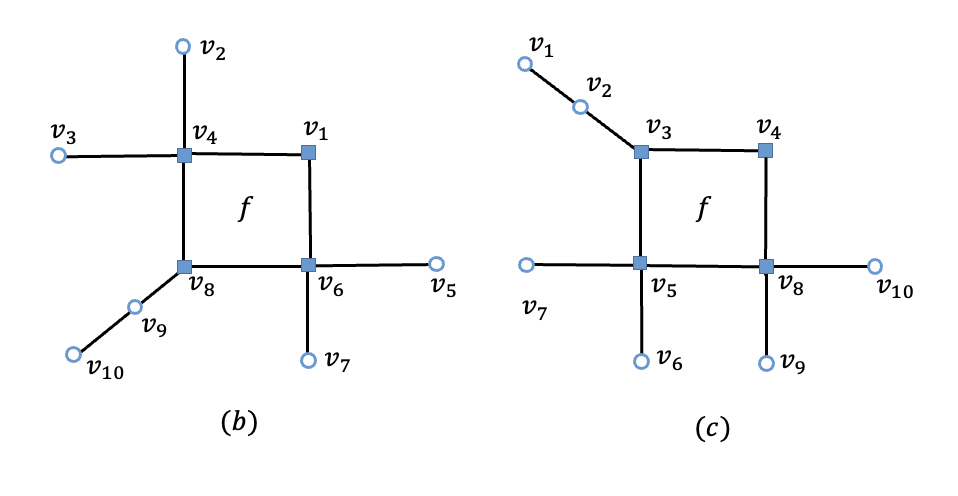}
\includegraphics[height=4.7cm,width=12cm]{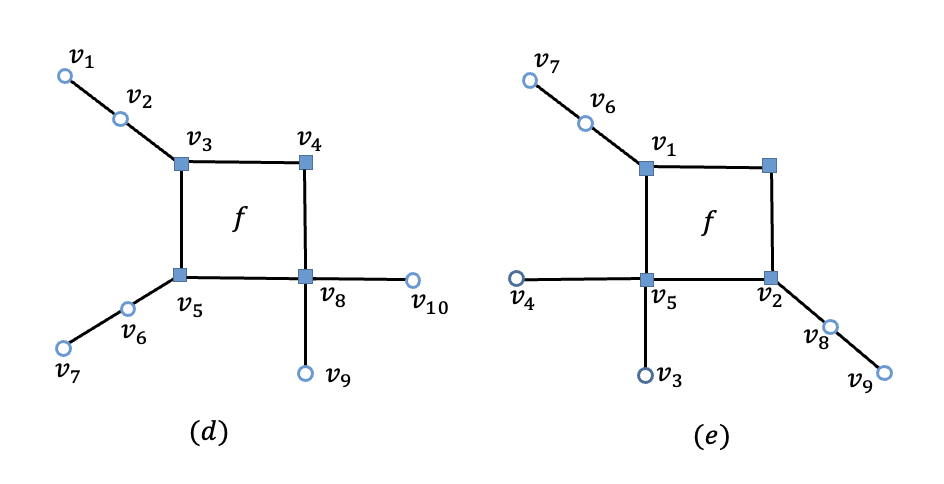}
\caption{weak 4-vertices in (4, 4, 4, 4)-face}\label{fig4}

\end{figure}

\end{lemma}

	\begin{proof}
Assume $f$ has three weak vertices and at least one vertex in $f$ is incident to a $(4,3,5,3)$-face. Then $G$ contains one of the graphs in Figure \ref{fig4} as a subgraph.   Since $G$ is triangle free and without normally adjacent 4-faces, then $G'$ is an induced subgraph of $G$. Assume $\mathcal{H'}=(L',H')$ is an $f'$-cover of $G'$.  We construct an $(\mathcal{H'}, 2m)$-coloring $\phi$ of $G'$ for each graph in Figure \ref{fig4}.

Assume $G'=G[\{v_1,v_2,v_3,v_4,v_5,v_6,v_7,v_8,v_9\}]$ is the subgraph in Figure \ref{fig4} (a). Choose an $m$-subset $S(v_9)$ from $L'(v_9)-N_{H'}(L'(v_7))-N_{H'}(L'(v_8))$ and add it to $\phi(v_9)$. 

Let $\mathcal{H''} = \mathcal{H'}-N_{H'}[S(v_9)]$. It suffices to prove $G'$ has an $(\mathcal{H''}, g)$-coloring $\phi$, where $g(v_9)=m$ and $g(v_i)=2m$ for $i\in[8]$. By Corollary \ref{(3m,4m,3m)-path}, $v_1v_2v_3$ has an $(\mathcal{H''}, 2m)$-coloring $\phi_1$. Similarly, $v_4v_5v_6$ has an $(\mathcal{H''}, 2m)$-coloring $\phi_2$. Add 
an $m$-subset of $L''(v_9) - N_H(\phi_1(v_2) \cup \phi_2(v_5))$ to $\phi(v_9)$, and then for $i=7,8$, color $v_i$ by $2m$-colors from $L(v_i) - N_H(\phi(v_9))$, we obtain an $(\mathcal{H}', 2m)$-coloring of $G'$.

Assume $G'$ is the graph in Figure \ref{fig4} (b).  Let $\mathcal{H''} = \mathcal{H'}-N_{H'}(v_1)$ be an $f''$-cover of $G[\{v_5,v_6, v_7\}]$. Thus $f''(v_6)= |L'(v_6)- N_{H'}(L'(v_1))|= 4m$. By Corollary \ref{(3m,4m,3m)-path}, the 3-path $v_5v_6v_7$ has an $(\mathcal{H''}, 2m)$-coloring $\phi_1$. 

Let $\mathcal{H'''} = \mathcal{H''}-N_{H''}(\phi_1(v_6))$ be an $f'''$-cover of $G[\{v_8,v_9, v_{10}\}]$. By Corollary \ref{(3m,4m,3m)-path}, the 3-path $v_8v_9v_{10}$ has an $(\mathcal{H'''}, 2m)$-coloring $\phi_2$. Then $\mathcal{H'''}$ is an $f'''$-cover of   $G''=G[\{v_1,v_2,v_3,v_4\}]$, where $f'''=(3m,3m,3m,5m)$. It follows from Lemma \ref{lem-4-vertex} that  $G''$ is $(f''',2m)$-DP-colorable.

 Cases (c) and (d) follow from Corollary \ref{cor-4p+3p+3p}.

Assume $G'=G[\{v_1,v_2,v_3,v_4,v_5,v_6,v_7,v_8,v_9\}]$ in Figure \ref{fig4} (e).  Let $G'_1=G\{v_1, v_6,v_7,v_2, v_8,v_9\}$. By lemma \ref{lem-key}, $(f', m)_{\{v_1, v_2\}} \preceq (f',2m)_{G'_1}$. Apply Lemma \ref{lem-BB} to $G'$, it suffices to show that $G'_2=G[\{v_1,v_2,v_3,v_4,v_5\}]$ is $(f'_2, g'_2)$-DP-colorable, where $f'_2 = (2m, 2m, 3m, 3m,5m)$, $g'_2 = (m, m, 2m, 2m,2m)$. Apply Corollary \ref{cor-key} to the 3-path $v_3v_5v_4$ with $k = m$, it suffices to show that $G'_2$ is $(f''_2, g''_2)$-DP-colorable, where $f''_2 = (2m, 2m, 2m, 2m, 4m)$ and $g''_2 =(m, m, m, m, 2m)$. This follows from Lemma \ref{4star}.
	\end{proof}

	We shall use discharging method to derive a contradiction. 
	Set the \emph{initial charge} $ch(v)=2d(v)-6$ for every $v\in G$, $ch(f)=d(f)-6$ for every face $f$. By Euler formula, 
	$$\sum_{x\in V(G)\cup F(G)}ch(x)<0.$$
	
	Denote by $\omega(v\rightarrow f)$  the charge transferred from a vertex $v$ to an incident face $f$. Below are the discharging rules:
	
\begin{enumerate}[R1]
\item Each strong $4$-vertex sends $\frac{2}{3}$ to each incident $4$-face and $\frac{1}{3}$ to each incident $5$-face.
\item Each normal $4$-vertex sends $1$ to the incident light $4$-face and $\frac{1}{3}$ to each incident $5$-face.
\item If $v$ is a  weak $4$-vertex and $f$ is $4$-face or $5$-face incident to $v$, then
		$$ \omega (v\rightarrow f)=\left \{ 
		\begin{aligned}
		1,  &~~\text{if $f$ is a light $4$-face,}\\
		\frac{1}{2},  &~~\text{if $f$ is a non-light $4$-face and $v$ is incident to at most one special $5$-faces,} \\
		\frac{1}{3},  &~~\text{if $f$ is a special $5$-face of $v$; or $f$ is a non-light $4$-face } \\
		&~~\text{ and $v$ is incident to two special $5$-faces,}\\
		\frac{1}{6},  &~~\text{if $f$ is a non-special $5$-face.} 
		\end{aligned} 
		\right.$$
		
		\item Assume $v$ is a  very weak $4$-vertex and $f$ is $4$-face or $5$-face incident to $v$.
		\begin{itemize}
			\item(i)  If $v$  incident to a $(4, 4, 4, 3)$-face, then
			$$ \omega (v\rightarrow f)=\left \{ 
			\begin{aligned}
			1,  &~~\text{if $f$ is a light $4$-face,}\\
			\frac{2}{3},  &~~\text{if $f$ is a $(4, 4, 4, 3)$-face,}\\
			\frac{1}{3},  &~~\text{if $f$ is a special $5$-face of $v$,}\\
			0,  &~~\text{if $f$ is   a non-special $5$-face of $v$.}\\
			\end{aligned} 
			\right.$$
			\item(ii) Otherwise, 
			$$ \omega (v\rightarrow f)=\left \{ 
			\begin{aligned}
			1,  &~~\text{if $f$ is a light $4$-face,}\\
			\frac{1}{3},  &~~\text{if $f$ is a $5$-face, or a non-light $4$-face.}\\
			\end{aligned} 
			\right.$$
		\end{itemize}

		\item Each $5$-vertex sends 1 to each incident $4$-face and sends $\frac{2}{3}$ to each incident $5$-face.
		\item Each $6^+$-vertex sends $\frac{4}{3}$ to each incident $4$-face and sends $\frac{2}{3}$ to each incident $5$-face.
	\end{enumerate}	
	
	\begin{obs}
		\label{ob1}
		If  $v$ is a very weak $4$-vertex  incident to a $5$-face $f$ and 
		$w(v \to f)=0$, then $v$ has a $5$-neighbor in $f$. 
	\end{obs}
	\begin{proof}
		Since $v$ is very weak and $w(v \to f)=0$, $v$ is incident to a light face and a $(4,4,4,3)$-face. By Lemmas \ref{(4,4,3,3)-face} and \ref{(4,3,5,3)-face}, the light face is a $(4,5,3,3)$-face.
		Since $w(v \to f)=0$, $f$ is not special, hence the neighbor of $v$ shared by $f$ and the light face is a $5$-vertex.
		\end{proof}

	Let $ch^*$ denote the final charge after  performing the discharging process. It suffices to show that the final charge of each vertex and each face is non-negative. 
	
	We first check the final charge of vertices in $G$.

	If $d(v)=3$, $ch^*(v)= ch(v)=0$. 
	
	If $v$ is a strong $4$-vertex, then since $v$ is incident to at most two $4$-faces, by R1,  $ch^*(v)  \ge ch(v)  -2\times \frac{2}{3}-2\times \frac{1}{3}=0$.

	If $v$ is a normal $4$-vertex, then by  R2, $ch^*(v) \ge ch(v) -1-3\times\frac{1}{3}=0$. 
	
	Assume $v$ is a weak $4$-vertex. If $v$ is incident to two special $5$-faces, then by R3, 
	$ch^*(v) \ge ch(v) -1-3\times\frac{1}{3}=0$. 
	
	If $v$ is incident to at most one special $5$-faces, $ch^*(v)\ge    ch(v) -1-\frac{1}{2}-\frac{1}{3}-\frac{1}{6}=0$.

	Assume that $v$ is a very weak $4$-vertex. If $v$ is incident to a $(4, 4, 4, 3)$-face, then by Lemmas \ref{(4,4,3,3)-face} and \ref{(4,3,5,3)-face}, $v$ is incident to a $(4, 5, 3, 3)$-face. Thus there is at most one  special $5$-face of $v$. By R4 (i), $ch^*(v)\ge    ch(v) -1-\frac{2}{3}-\frac{1}{3}=0$. Otherwise, by R4 (ii), $ch^*(v) \ge ch(v)  -1-3\times\frac{1}{3}=0$.
	
	If $d(v)=5$, then $v$ is incident at most two $4$-faces and by R5, $ch^*(v)\ge ch(v) -2\times 1-3\times \frac{2}{3}=0$.
	
	If $d(v)=k\ge 6$, then $v$ is incident at most $\lfloor \frac{k}{2} \rfloor$ $4$-faces. Thus by R6, $ch^*(v)\ge ch(v) -\frac{4}{3}\times\lfloor \frac{k}{2} \rfloor-(k-\lfloor \frac{k}{2} \rfloor)\times \frac{2}{3}\ge 0$.
	
	Now we check the final charge of faces. If $f$ is a $6^+$-face, no charge is discharged from or to $f$. Thus $ch^*(f) = ch(f) = d(f)-6 \ge0$.

	Assume $f$ is a $4$-face. 
	By Corollary \ref{cor-reducible} (a),   $f$ contains at most two $3$-vertices. 
	
	\medskip
	\noindent
	{\bf Case 1}   $f$ contains two $3$-vertices. 
	
	Assume $f$ contains a $6^+$-vertex. If $f$ contains a $4$-vertex $v$, then by Lemma \ref{4-faces with two 3-vertices}, $v$ is a strong $4$-vertex. Hence $f$ receives $\frac{4}{3}$ from the $6^+$-vertex by R5 and at least  $\frac{2}{3}$ from the other $4^+$-vertex by R1, R5 and R6. So $ch^*(f) \ge 0$.

	If $f$ contains two $5$-vertices, then  $f$ receives 1 from each incident $5$-vertex by $R5$, and hence $ch^*(f) \ge 0$.
	
	Otherwise, $f$ is a light $4$-face, and receives 1 from each incident $4^+$-vertex by R2-R5, and hence $ch^*(f) \ge 0$.

	\medskip
	\noindent
	{\bf Case 2}  $f$ contains one 3-vertex. 	
	
	If $f$ contains no very weak $4$-vertex, then every $4^+$-vertex in $f$ sends at least  $\frac{2}{3}$ to $f$ by R1, R5 and R6. Thus $ch^*(f)\ge ch(f) +3\times\frac{2}{3}=0$. 
	
	Assume that $f$ contains a very weak $4$-vertex. If $f$ is $(4, 4, 4, 3)$-face, $ch^*(f)\ge ch(f) +3\times\frac{2}{3}=0$ by R1 and R4 (i). Assume that $f$ is not a $(4, 4, 4, 3)$-face. Then $f$ contains a $5^+$-vertex. By Lemma \ref{(4,4,4,3)-face}, $f$ contains at most one very weak $4$-vertex.
	Thus $ch^*(f)\ge ch(f)+1+\frac{2}{3}+\frac{1}{3}=0$ by R1, R4 (ii) and R5.

\medskip
	\noindent
	{\bf Case 3} $f$ contains no  $3$-vertex. 
	
	Assume $f$ is $(4, 4, 4, 4)$-face. If  no vertex of $f$  is incident to $(4, 3, 5, 3)$-face, then each 
	   vertex $v$ of $f$ has at most one $3$-neighbor and hence has at most one special $5$-face. So
	   $ch^*(f)\ge ch(f)+4\times\frac{1}{2}=0$ by R3. 
	
	 If   $f$ has a vertex  $v$ incident to a $(4, 3, 5, 3)$-face, then $f$ contains at most two weak vertices by Lemma \ref{(4,4,4,4)-face}. Thus $ch^*(f)\ge ch(f)+2\times\frac{1}{3}+2\times\frac{2}{3}=0$ by R1 and R3.  
	
	Assume $f$ is $(4^+, 4^+, 4^+, 5^+)$-face.  Then   $ch^*(f) \ge ch(f)+1+3\times \frac{1}{3}=0$ by R3 and R5. 
	
	This completes the check for $4$-faces. 
	
	\medskip
	
Finally , we check the $5$-faces. 

Assume $f=(v_1, v_2, v_3,v_4, v_5)$ is a 5-face, and for $i=1,2,3,4,5$, let $f_i$ be the face sharing the edge $v_iv_{i+1}$ with $f$ (the indices are modulo $6$).

		By Corollary \ref{cor-reducible}, either $f$ contains at least three $4^+$-vertices or $f$ contains two $4^+$-vertices and one of them is a  $5^+$-vertex. 
		
		If $f$ contains no weak and no very weak $4$-vertex, or $f$ is a special $5$-face, then $f$ receives at least $\frac{1}{3}$ from each incident $4$-vertex and $\frac{2}{3}$ from each incident $5^+$-vertex by R1-R5. Hence $ch^*(f) \ge ch(f)+1 =0$.
		
		Assume $f$ is a non-special $5$-face and $f$ contains a weak or a very weak $4$-vertex.

			\medskip
			\noindent
			{\bf Case 1} $f$ contains a weak 4-vertex. 
			
			Assume $v_1$ is a weak $4$-vertex. By symmetry, we may assume that $f_5$ is a light $4$-face
			and $f_1$ is a $4$-face with no $3$-vertex. Thus $v_2$ is a $4^+$-vertex. 
			
			If $f_5$ is a $(4, 5, 3, 3)$-face, then since $f$ is non-special,  $v_5$ is a $5$-vertex. Then $w(v_5 \to f) = 2/3$ and $w(v_i \to f) \ge 1/6$ for $i=1,2$. So $ch^*(f) \ge ch(f)+1 =0$. 
			
			Assume $f_5$ is a $(4,4,3,3)$-face. Each of $v_1, v_5$ sends at least $1/6$ to $f$. If $f$ contains a $5^+$-vertex, then
			$ch^*(f) \ge ch(f)+1 =0$.  Assume $f$ contains no $5^+$-vertex. So by Corollary \ref{cor-reducible},  $v_2$ and $v_4$ are $4$-vertices.   
			
			By Lemma \ref{4-faces with two 3-vertices}, none of $f_1$ and $f_4$ is a  light $4$-face. If $v_3$ is a $3$-vertex, then each of $v_2$ and $v_4$ sends $1/3$ by R1-R4. Hence $ch^*(f) \ge ch(f)+1 =0$. 
			
			Assume $v_4$ is a $4$-vertex. Then $f$ is a $(4,4,4,4,4)$-face. By Observation 
			\ref{ob1}, each $4$-vertex sends at least $1/6$ to $f$. As $f$ is adjacent to at most two light $4$-faces, at least one of the $4$-vertex sends  $1/3$ to $f$. Hence $ch^*(f) \ge ch(f)+1 =0$.

		\medskip
	\noindent
	{\bf Case 2} $f$ contains no weak vertex and contains a   very weak $4$-vertex.

	Assume  $v_1$ is a very weak vertex, $f_5$ is a light $4$-face and $f_1$ is a $4$-face containing one $3$-vertex. 
  Note that 
	$f_5$  is not a $(4, 3, 5, 3)$-face, for otherwise, $f$ is a special 5-face of $v_1$.

	Assume first that $f_1$ is  a $(4, 4, 4, 3)$-face.
	By Lemma \ref{(4,4,3,3)-face}, $f_5$ is a $(4,5,3,3)$-face. 
	Hence $v_5$ is a $5$-vertex. If $v_2$ is a $4$-vertex, then
	$w(v_5 \to f) = 2/3$ and $w(v_2 \to f) =1/3$. Hence $ch^*(f) \ge ch(f)+1 =0$. If $v_2$ is a $3$-vertex, then $f_2$ is not a $4$-face. If $v_3$ is a $3$-vertex, then $G$ contains a $(3,3,4,3,3)$-path, which is reducible. Thus $v_3$ is a $4^+$-vertex and is not weak or very weak. So $w(v_3 \to f) \ge 1/3$ and $ch^*(f) \ge ch(f)+1 =0$.
	
	Assume $f_1$ is not a $(4, 4, 4, 3)$-face. Since $f$ contains no weak $4$-vertex,   each $4$-vertex of $f$ sends at least $1/3$ to $f$ and each $5^+$-vertex sends at least $2/3$ to $f$. Hence  $ch^*(f) \ge ch(f)+1 =0$. 
	
	This completes the proof of Theorem \ref{thm-a}.


\end{document}